\documentclass{article}
\usepackage[utf8]{inputenc}
\usepackage{amsmath,amssymb,amsthm,amsfonts,latexsym,graphicx,subfigure,mathtools,enumitem}

\usepackage{url}
\usepackage{xcolor}
\usepackage{svg}

\newcommand{\R}{\mathbb{R}}
\newcommand{\RQ}{\textrm{RQ}}

\DeclareMathOperator{\id}{Id}

\DeclareMathOperator{\vol}{vol}

\theoremstyle{plain}
\newtheorem{theorem}{Theorem}[section]
\newtheorem{proposition}[theorem]{Proposition}

\newtheorem{corollary}[theorem]{Corollary}
\newtheorem*{corollary*}{Corollary}

\newtheorem{question}{Question}

\theoremstyle{definition}
\newtheorem{definition}[theorem]{Definition}
\newtheorem{example}[theorem]{Example}
\theoremstyle{remark}
\newtheorem{remark}[theorem]{Remark}

\usepackage{stmaryrd} 
\usepackage[affil-it]{authblk}
\usepackage{booktabs}   
\usepackage[numbers,sort&compress]{natbib}

\usepackage{comment} 

\allowdisplaybreaks

\usepackage{showkeys}
\usepackage{hyperref}

\usepackage[capitalise,noabbrev]{cleveref}

\usepackage[margin=3.5cm]{geometry}

\title{Fair and Tolerant (FAT) Graph Colorings}

\author{Lies Beers\thanks{e.g.m.beers@vu.nl} }
\author{Raffaella Mulas}
\date{}
\affil{Vrije Universiteit Amsterdam}
\begin{document}

\maketitle

\begin{abstract} 
We introduce and study Fair and Tolerant colorings (FAT colorings), where each vertex tolerates a given fraction of same-colored neighbors while fairness is preserved across the other coloring classes. Moreover, we define the FAT chromatic number $\chi^{\mathrm{FAT}}(G)$ as the largest integer $k$ for which $G$ admits a FAT $k$-coloring. We establish general bounds on $\chi^{\mathrm{FAT}}$, relate it to structural and spectral properties of graphs, and characterize it completely for several families of graphs. We conclude with a list of open questions that suggest future directions.
\end{abstract}

\vspace{0.2cm}

\section{Introduction}

\subsection*{Historical note}

In 1888, in a letter to his sister Willemien, Vincent van Gogh wrote \cite{VanGogh1888}: 

\begin{quote}
        \emph{One can speak poetry just by arranging colors well, just as one can say comforting things in music.}
\end{quote}

Van Gogh was referring to painting, but his reflection also captures the essence of graph coloring, whose history began 36 years earlier. \newline 

In fact, while graph theory itself was born in 1736 with Leonhard Euler's solution to the Königsberg bridges problem \cite{Euler1736}, the first problem in graph coloring was formulated in 1852, when Francis Guthrie conjectured that four colors are sufficient to color any map so that no two adjacent regions share the same color \cite{Voloshin2009,Jensen2011,Maritz2012,Wilson2013}. Guthrie's conjecture soon became one of the most famous open problems in mathematics, until 1976, when Kenneth Appel and Wolfgang Haken confirmed its validity with the first major computer-assisted proof in mathematics \cite{Appel1977}. Since then, classical graph coloring and its many variants have developed into central topics in graph theory, inspiring deep theorems and a wide range of applications. \newline 

Today, echoing Van Gogh's words, graph theorists know that arranging colors well is indeed a form of poetry. 

\subsection*{The idea behind FAT colorings}

Classical graph coloring requires that adjacent vertices always receive different colors. Here, we relax this rule by allowing a controlled amount of \emph{tolerance}: a vertex may share its color with a fixed fraction of its neighbors, while the remaining neighbors are distributed \emph{fairly} among the other colors. Hence the name \emph{Fair and Tolerant} (FAT) colorings. \newline

This framework captures situations where strict separation is not necessary, but balance is essential. For example, in social or biological networks, vertices can represent individuals, edges can represent interactions, and colors can represent groups or traits: \emph{tolerance} corresponds to homophily (the tendency of similar individuals to connect), while \emph{fairness} reflects diversity across different groups.

\subsection*{Structure of the paper}

In Section \ref{section:definitions} we introduce the formal definition of FAT colorings and establish their fundamental properties.  Section \ref{section:spectral} explores spectral aspects, relating the existence of FAT colorings to the eigenvalues of graph operators.  Moreover, in Section \ref{section:Turan} we characterize the FAT chromatic number of regular Turán graphs. Finally, Section \ref{section:irr} studies irreducible FAT colorings, while Section \ref{section:questions} collects a number of open questions that suggest future directions.

\section{Definitions and fundamental properties}\label{section:definitions}

Let $G=(V,E)$ be a simple graph, that is, an undirected, unweighted graph without multi-edges and without loops. Throughout the paper, we assume that $G$ is connected and has at least $2$ vertices. For a vertex $v\in V$ and a subset $S\subseteq V$, we let
$$
e(v,S) := |\{w\in S : v\sim w\}|
$$
denote the number of neighbors of $v$ contained in $S$. Similarly, given two (not necessarily distinct) subsets $S,T\subseteq V$, we let
$$
e(S,T) := |\{\{u,v\}\in E : u\in S, v\in T\}|
$$
be the number of edges with one endpoint in $S$ and the other in $T$.\newline 

Recall that a \emph{$k$-coloring} of $G$ is a function $c:V\to \{1,\ldots,k\}$, and for each $i\in\{1,\ldots,k\}$, the set $V_i := c^{-1}(i)$ is called the \emph{coloring class} of color $i$. Moreover, the coloring is \emph{proper} if $v\sim w$ implies $c(v)\neq c(w)$. The \emph{chromatic number} (or \emph{coloring number}) $\chi=\chi(G)$ is the smallest $k$ for which there exists a proper $k$-coloring of $G$.\newline 

As a convention, we consider two $k$-colorings to be the same if one can be obtained from the other by a permutation of the coloring classes.\newline

In the next definition, we relax the requirement of properness, allowing each vertex to tolerate a fixed fraction of neighbors in its own coloring class. Moreover, we require that the remaining neighbors are distributed evenly among the other classes.

\begin{definition}
    A vertex $k$-coloring with non-empty classes $V_1,\ldots,V_k$ is Fair and Tolerant (FAT) if there exists $\alpha\in [0,1]$ such that
    $$
    e(v,V_i)=\begin{cases}
        \alpha \deg v, &\text{if }v\notin V_i,\\
        \beta \deg v, &\text{if }v\in V_i,
    \end{cases}
    $$ where $\beta:=1-(k-1)\alpha$. We assume that, if $k=1$, then $\alpha=0$.
\end{definition}

Note that $\beta \in [0,1]$ is defined so that $1=(k-1)\alpha+\beta$, and therefore $$\deg v=(k-1)\alpha \deg v+\beta\deg v, \quad \text{for all } v\in V(G).$$ 
Hence, the definition of $\beta$ guarantees that all neighbors of $v$ are accounted for.\newline 

The terminology reflects the two principles behind FAT colorings. \emph{Fairness} means that every vertex assigns the same fraction $\alpha$ of its neighbors to each of the other $k-1$ coloring classes, while \emph{tolerance} means that a remaining fraction $\beta$ of its neighbors are allowed to share its own color (Figure \ref{fig:FATillustration}).

\begin{figure}[h]
    \centering
    \includegraphics[width=0.5\linewidth]{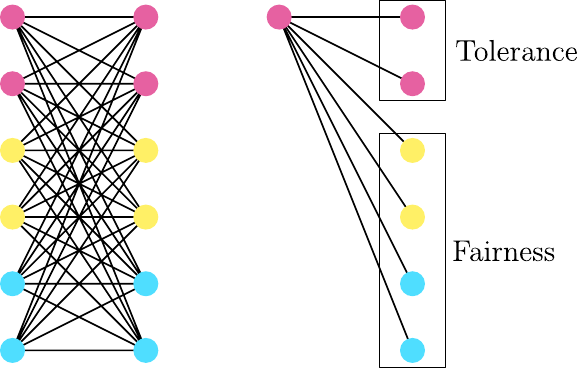}
    \caption{An illustration of the concepts of fairness and tolerance.}
    \label{fig:FATillustration}
\end{figure}

\begin{remark}\label{rmk:proper}
The case $\beta=0$ corresponds to colorings with no tolerance, that is, every vertex has no neighbors in its own coloring class. In other words, FAT colorings with $\beta=0$ are precisely proper $k$-colorings. In this case, we have $\alpha=1/(k-1)$. Such colorings are sometimes referred to as \emph{equitable} in the literature \cite{beersmulas2024}. However, since the word ``equitable'' is also used with different meanings in graph coloring, here we adopt the terminology FAT to avoid confusion.
\end{remark}

\begin{remark}
FAT colorings can be viewed as a special case of broader notions that have been studied in the literature, including \emph{domatic colorings} (colorings that are not necessarily proper, where each vertex is connected to at least one vertex of every color that is different from its own one) \cite{zelinka1981domatic}, and \emph{majority colorings} (where the parameters $\alpha$ and $\beta$ are not necessarily constant for all vertices, and the requirement is that $e(v,V_i)\leq \deg v/2$ whenever $v\in V_i$) \cite{bosek2019majority}. What distinguishes FAT colorings is the fixed fractional prescription: each vertex has exactly the same fraction $\alpha$ of neighbors in every other class, and fraction $\beta$ in its own class. 
\end{remark}

Now, note that every graph admits a trivial FAT $1$-coloring in which all vertices are placed in a single coloring class. For connected graphs, this FAT coloring corresponds to the parameters 
$$
\alpha = 0, \quad \beta = 1.
$$

It is therefore natural to ask, for a given graph $G$, what is the largest integer $k$ such that $G$ admits a FAT $k$-coloring. 

\begin{definition}
    The  \emph{FAT chromatic number} of a graph $G$, denoted by $\chi^{\textrm{FAT}}(G)$ or simply $\chi^{\textrm{FAT}}$, is the largest integer $k$ such that  $G$ admits a FAT $k$-coloring. 
\end{definition}

\begin{example}[Complete graphs]\label{ex:KN}
    For the complete graph $K_N$ on $N$ vertices, if each vertex is assigned a distinct color, then this gives a FAT coloring with parameters
$$
\beta = 0, \quad \alpha = \tfrac{1}{N-1}.
$$
Hence,
$$
\chi^{\mathrm{FAT}}(K_N) = \chi(K_N) = N.
$$
\end{example}

\begin{example}\label{ex:bipartite>=2}
For any bipartite graph, $\chi^{\mathrm{FAT}}\geq 2$, since any proper $2$-coloring is also a FAT coloring with parameters $\alpha=1$ and $\beta=0$.
\end{example}

\begin{remark}\label{rmk:merging}
Suppose that $G$ admits a FAT $k$-coloring and that $k = c \bar{k}$ for some $c,\bar{k} \in \mathbb{N}$, 
i.e., $\bar{k}$ divides $k$. Then, $G$ also admits a FAT $\bar{k}$-coloring, obtained by merging the $k$ coloring classes into $\bar{k}$ groups of $c$ classes each, as we shall prove in Theorem \ref{thm:merging} below. We note that this divisibility property is not unique to FAT colorings: it also appears in the context of \emph{oriented edge periodic colorings}, see \cite[Theorem 4]{maximal-colouring}.
\end{remark}

\begin{proposition}\label{prop:delta+1}
Let $G$ be a connected graph, and let $\delta$ be the minimum vertex degree in $G$. If $G$ admits a FAT $k$-coloring, then $k\leq \delta+1$. In particular, $$\chi^{\textrm{FAT}}\leq \delta+1,$$ and the bound is sharp.
\end{proposition}
\begin{proof}
If $G$ admits a FAT $k$-coloring, then clearly $\deg v\geq k-1$ for all $v\in V$. Therefore, $k\leq \delta+1$ and, in particular, $\chi^{\textrm{FAT}}\leq \delta+1$. The fact that the latter bound is sharp follows from Example \ref{ex:KN}.
\end{proof}

An immediate consequence of Proposition \ref{prop:delta+1} is the following.

\begin{corollary}
    For every connected graph $G$ on $N$ vertices,
    $$\chi^{\textrm{FAT}}\leq N,$$
    with equality if and only if $G=K_N$.
\end{corollary}

\begin{proof}
    The claim follows from Example \ref{ex:KN}, from Proposition \ref{prop:delta+1} and from the fact that $\delta\leq N-1$, with equality if and only if $G=K_N$.
\end{proof}

\begin{example}[Trees]\label{ex:trees}
Let $G$ be a connected bipartite graph with minimum degree $\delta=1$ (for instance, the path graph, or more generally any tree). By Proposition \ref{prop:delta+1}, $\chi^{\mathrm{FAT}} \leq 2$. 
Moreover, since $\chi(G)=2$ and any proper $2$-coloring is also a FAT coloring (with parameters $\alpha=1$ and $\beta=0$), it follows that $\chi^{\mathrm{FAT}}=2$. Thus, the bound in Proposition \ref{prop:delta+1} holds with equality in this case.
\end{example}

\begin{figure}[h]
    \centering
    \includegraphics[width=4.5cm]{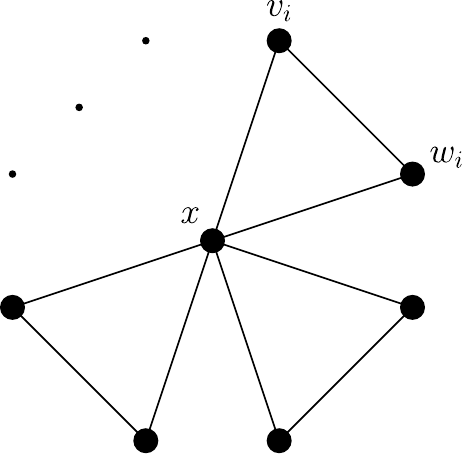}
    \caption{The petal graph}
    \label{fig:petal}
\end{figure}

\begin{example}[Petal graphs]\label{ex:petalgraphs}
Given $m\geq 1$, the $m$--\emph{petal graph} is the graph on $N=2m+1$ nodes such that (Figure \ref{fig:petal}):
\begin{itemize}
    \item $V=\{x,v_1,\ldots,v_m,w_1,\ldots,w_m\}$;
    \item $E=\{(x,v_i)\}_{i=1}^m\cup \{(x,w_i)\}_{i=1}^m\cup\{(v_i,w_i)\}_{i=1}^m$.
\end{itemize}

In this case, $\delta=2$, therefore Proposition \ref{prop:delta+1} gives $\chi^{\mathrm{FAT}} \leq 3$. Now, observe that there exists a FAT $3$-coloring of the vertices, obtained for instance by assigning color $1$ to $x$, color $2$ to all vertices $v_1,\ldots,v_m$, and color $3$ to all vertices $w_1,\ldots,w_m$. This coloring satisfies the FAT conditions with $\alpha=1/2$ and $\beta=0$. Hence, $\chi^{\mathrm{FAT}} = 3$, and the bound in Proposition \ref{prop:delta+1} holds with equality also in this case.

\end{example}

In regular graphs, FAT colorings impose strong symmetry on the coloring classes:

\begin{theorem}\label{thm:regular}
    Let $G$ be a connected regular graph on $N$ vertices. For any FAT $k$-coloring of $G$, all coloring classes have the same size. In particular, each class has size $N/k$, and therefore $k$ divides $N$.
\end{theorem}

\begin{proof}

Assume that $G$ is $d$-regular, and consider a FAT $k$-coloring with classes $V_1,\ldots,V_k$. If $k=1$, then there is only one coloring class and the claim is trivial. If $k>1$, then $\alpha\neq 0$ since $G$ is connected. In this case, for $i\neq j$ we have
$$
e(V_i,V_j)=\sum_{v\in V_i} e(v,V_j)=|V_i| \cdot \alpha \cdot d,
$$ and similarly
$$
e(V_i,V_j)=\sum_{w\in V_j} e(w,V_i)=|V_j| \cdot \alpha \cdot d.
$$
Therefore, $$
|V_i| \cdot \alpha \cdot d=|V_j| \cdot \alpha \cdot d.
$$ Since $\alpha \neq 0$, this implies that $|V_i|=|V_j|$. Hence each class has size $N/k$, and $k$ divides $N$.
\end{proof}

\begin{remark}
    The statement of Theorem \ref{thm:regular} relies on both regularity and connectedness. \newline
For non-regular graphs, the coloring classes of a FAT $k$-coloring do not need to have the same size. For instance, in the star graph $K_{1,n}$, the center and the leaves can belong to classes of different cardinalities while still satisfying the FAT conditions (Figure \ref{fig:star}).

\begin{figure}[h]
    \centering
    \includegraphics[scale=0.66]{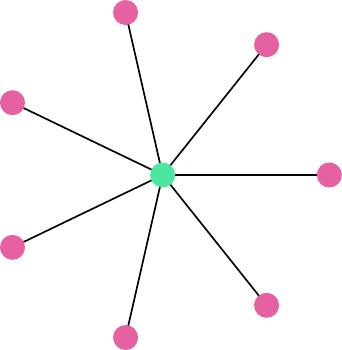}
    \caption{A FAT $2$-coloring of the star graph $K_{1,7}$.}
    \label{fig:star}
\end{figure}

To see that connectedness is also necessary, consider a disconnected graph obtained as the disjoint union of two $d$-regular graphs with different orders, say $G=G_1 \sqcup G_2$ with $|V(G_1)|\neq |V(G_2)|$. If we color each component monochromatically and let $V_1=V(G_1)$ and $V_2=V(G_2)$ be the coloring classes, then every vertex has all neighbors in its own coloring class. Therefore, this gives a FAT $2$-coloring with $\alpha=0$ and $\beta=1$, in which the coloring classes have different sizes.
\end{remark}

We now consider non-regular graphs. In the particular case where the greatest common divisor of the vertex degrees is equal to $1$, the possibilities for FAT colorings are very restricted, as the next theorem shows.

\begin{theorem}\label{thm:gcd}
    Let $G$ be a connected graph such that
    $$
    \gcd \{\deg v : v \in V\}=1.
    $$ Then, exactly one of the following two cases holds:

    \begin{enumerate}
    \item $\chi^{\textrm{FAT}}=1$ and $G$ is not bipartite.
        \item $\chi^{\textrm{FAT}}=2$ and $G$ is bipartite.
    \end{enumerate}
\end{theorem}

\begin{proof} Fix  a FAT $\chi^{\textrm{FAT}}$-coloring of $G$, with parameter $\alpha$. We consider two cases.

\begin{itemize}
    \item Case 1: $\alpha=0$. Since $G$ is connected, this implies that $\chi^{\textrm{FAT}}=1$. In particular, $G$ cannot be bipartite, as otherwise it would admit a (proper) FAT $2$-coloring, and this would imply that $\chi^{\textrm{FAT}}>1$.
    \item Case 2: $\alpha \neq 0$. In this case, $\chi^{\textrm{FAT}}>1$. Consider the reduced form of $\alpha$, that is, write $\alpha=r/s$, with $r,s\in \mathbb{N}$ and $\gcd(r,s)=1$. Note that this is always possible by definition of FAT coloring. Then, given $v\in V$ and a coloring class $V_i$ such that $v\notin V_i$, 
    $$
     e(v,V_i)= \alpha \deg v = \frac{r}{s}\cdot  \deg v.
    $$ Hence, $s \cdot e(v,V_i)=r \cdot \deg v$. Since $\gcd(r,s)=1$, this implies that $s|\deg v$, for all $v\in V$. Thus,
    $$
    s|\gcd \{\deg v : v \in V\}=1,
    $$ implying that $s=1$. Since $\alpha=r/s\in[0,1]$, it follows that $\alpha=1$. As a consequence, $\beta=0$. Therefore, $\chi^{\textrm{FAT}}=2$ and the coloring is proper, implying that $G$ is bipartite.
\end{itemize}

\end{proof}

\begin{figure}[h]
    \centering
    \includegraphics[width=6.7cm]{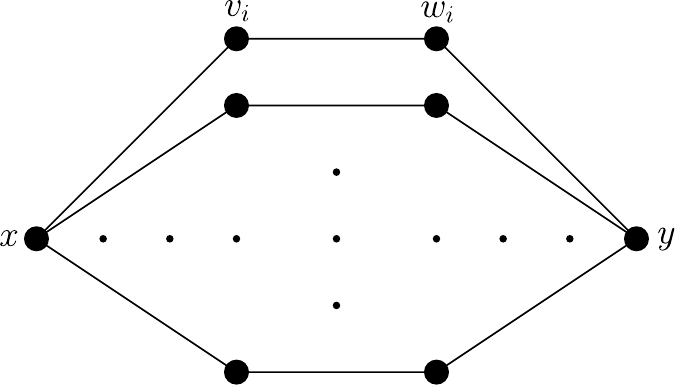}
    \caption{The book graph}
    \label{fig:book}
\end{figure} 

\begin{example}[Book graphs]\label{ex:bookgraphs}

Given $m\geq 1$, the $m$--\emph{book graph} \cite{petalsbooks} is the graph on $N=2m+2$ nodes such that (Figure \ref{fig:book}):
\begin{itemize}
    \item $V=\{x,y,v_1,\ldots,v_m,w_1,\ldots,w_m\}$;
    \item $E=\{(x,v_i)\}_{i=1}^m\cup \{(y,w_i)\}_{i=1}^m\cup\{(v_i,w_i)\}_{i=1}^m$.
\end{itemize}

    When $m$ is odd, $$\gcd \{\deg v : v \in V\}=\gcd \{2,m\}=1$$ and, since the book graph is bipartite, Theorem \ref{thm:gcd} gives $\chi^{\textrm{FAT}}=2$ (Figure \ref{fig:bookcolored}).
    
    When $m$ is even, $$\gcd \{\deg v : v \in V\}=\gcd \{2,m\}=2,$$ therefore Theorem \ref{thm:gcd} does not apply. By Proposition~\ref{prop:delta+1}, we have $\chi^{\textrm{FAT}}\leq \delta+1=3$. Moreover, a FAT $3$-coloring exists with parameters $\alpha=1/2$ and $\beta=0$, obtained as follows.
\begin{align*}
    &\text{Color 1: } \{x,y\},\\
    &\text{Color 2: } \{v_1,\ldots,v_{m/2}\} \cup \{w_{m/2+1},\ldots,w_m\},\\
    &\text{Color 3: } \{w_1,\ldots,w_{m/2}\} \cup \{v_{m/2+1},\ldots,v_m\}.
\end{align*}
It is easy to check that every vertex has exactly half of its neighbors in each of the other two colors, and none in its own coloring class. Thus, $\chi^{\textrm{FAT}}=3$ (Figure \ref{fig:bookcolored}).
\end{example}

\begin{figure}[h]
    \centering
    \includegraphics[scale=0.66]{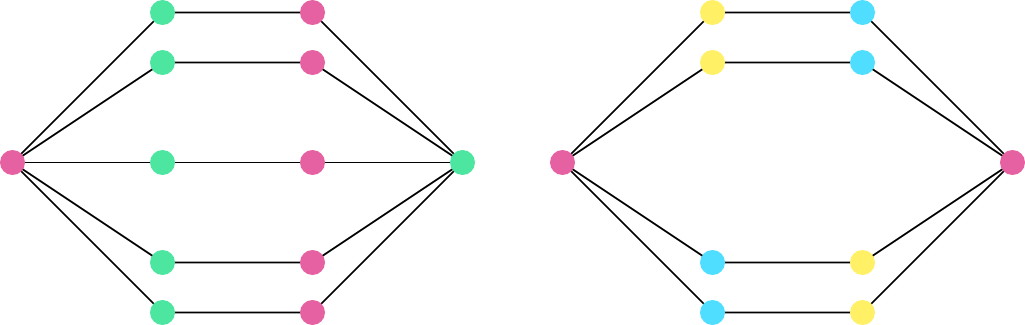}
    \caption{Optimal FAT colorings of book graphs for $m$ odd (left) and $m$ even (right).}
    \label{fig:bookcolored}
\end{figure}

\begin{remark}\label{rmk:chi}
We have seen examples with $\chi=\chi^{\mathrm{FAT}}$ (for instance complete graphs, trees, petal graphs, and $m$-book graphs when $m$ is odd), and examples with $\chi<\chi^{\mathrm{FAT}}$ (for instance $m$-book graphs when $m$ is even). 
There are also graphs with $\chi>\chi^{\mathrm{FAT}}$. For example, consider an odd cycle $C_N$ on $N$ nodes such that $N$ is not divisible by $3$. In this case, $\chi=3$, and by Proposition \ref{prop:delta+1}, $\chi^{\mathrm{FAT}}\leq \delta+1=3$. By Theorem \ref{thm:regular}, a FAT $k$-coloring with $k>1$ would require all coloring classes to have the same size, and this is impossible since $N$ is not divisible by $2$ or $3$. Hence, $\chi^{\mathrm{FAT}}(C_N)=1$ if $N$ is not divisible by $2$ or $3$, implying that $\chi>\chi^{\mathrm{FAT}}$ in this case.
\end{remark}

\begin{example}[Cycle graphs]\label{ex:cyclegraphs}
   By Example \ref{ex:KN}, the odd cycle $C_3$ on $3$ nodes satisfies
   $$\chi^{\mathrm{FAT}}(C_3)=3.$$
   Moreover, by Remark \ref{rmk:chi}, any odd cycle $C_N$ on $N$ nodes such that $N$ is not divisible by $3$ satisfies $$\chi^{\mathrm{FAT}}(C_N)=1.$$ 
   
   It is natural to ask what happens for the other cycles.

First, we consider any cycle $C_N$ on $N$ nodes such that $N$ is divisible by $3$. By Proposition \ref{prop:delta+1}, $\chi^{\mathrm{FAT}}\leq \delta+1=3$, and in this case, we can construct a FAT $3$-coloring by repeating the color sequence $(1,2,3)$ around the cycle. Each vertex has one neighbor in each of the two other coloring classes, therefore the fairness condition is satisfied. Thus, in this case,
     $$  
       \chi^{\mathrm{FAT}}(C_N)=3.
       $$   

   It is left to consider the case of even cycles $C_{2\ell}$ such that $3\nmid 2\ell$. Note that, by Example \ref{ex:bipartite>=2}, $\chi^{\mathrm{FAT}}(C_{2\ell})\geq 2$ since $C_{2\ell}$ is bipartite. Moreover, by Proposition \ref{prop:delta+1}, $\chi^{\mathrm{FAT}}(C_{2\ell})\leq 3$, since $\delta=2$ in this case. Hence,
   $$
   \chi^{\mathrm{FAT}}(C_{2\ell})\in \{2,3\}.
   $$

Now, by Theorem~\ref{thm:regular}, the FAT chromatic number of a regular graph must divide the number of its vertices. Since we are assuming that $3\nmid 2\ell$, a FAT $3$-coloring is impossible, hence necessarily
       $$
       \chi^{\mathrm{FAT}}(C_{2\ell})=2.
       $$
      In summary (Figure \ref{fig:cycles}), 
   $$
      \chi^{\mathrm{FAT}}(C_N)=\begin{cases}
          1, &\text{ if $N$ is odd and not divisible by $3$},\\
        2, &\text{ if $N$ is even and not divisible by $3$,} \\
        3, &\text{ if $N$ is a multiple of $3$.} 
      \end{cases}
      $$
   \end{example}

    We now prove a generalization of both Proposition \ref{prop:delta+1} and Theorem \ref{thm:gcd} by adapting the proof of the latter.
    \begin{theorem}\label{thm:gcdgen}
        Let $G$ be a connected graph.
        Then,
        \[
        \chi^{\textrm{FAT}} \leq \gcd\bigl\{\deg v \colon v\in V\bigr\} + 1.
        \]
        Moreover, if equality holds, then any $\chi^{\textrm{FAT}}$-coloring is a proper FAT coloring.
    \end{theorem}
    \begin{proof}
        Fix a FAT $\chi^{\textrm{FAT}}$-coloring of $G$ with parameter $\alpha$. As in the proof of Theorem \ref{thm:gcd}, write $\alpha = r/s$ with $r,s\in \mathbb N$ and $\gcd(r,s) = 1$. Then, for any vertex $v\in V$ and coloring class $V_i$ such that $v\notin V_i$, we have
        \[
        e(v,V_i) = \alpha \deg v = \frac rs \cdot \deg v.
        \]
        Hence, $$s\cdot e(v,V_i) = r\cdot \deg v.$$ Since $\gcd(r,s)=1$, it follows that $s \mid \deg v$ for all $v\in V$. Consequently,
        \[
        s \mid \gcd\bigl\{\deg v \colon v\in V\bigr\}.
        \]
        Now, because $\alpha= r/s$, any vertex can have neighbors in at most $\lfloor s/r\rfloor + 1$ distinct coloring classes. Hence,
        \[
        \chi^{\textrm{FAT}}\leq \left\lfloor \frac sr\right\rfloor + 1 \leq  s+1 \leq \gcd\bigl\{\deg v \colon v\in V\bigr\} + 1.
        \]
    
        If equality holds, then, in particular, $s = \lfloor s/r\rfloor = \chi^{\textrm{FAT}}-1$. Hence, $r=1$, and
        \[
        \alpha = \frac rs = \frac1{\chi^{\textrm{FAT}}-1}.
        \]
        We conclude that $\beta = 0$, and that any $\chi^{\textrm{FAT}}$-coloring is proper.
    \end{proof}

\begin{figure}[h]
   \centering
\includegraphics[scale=0.66]{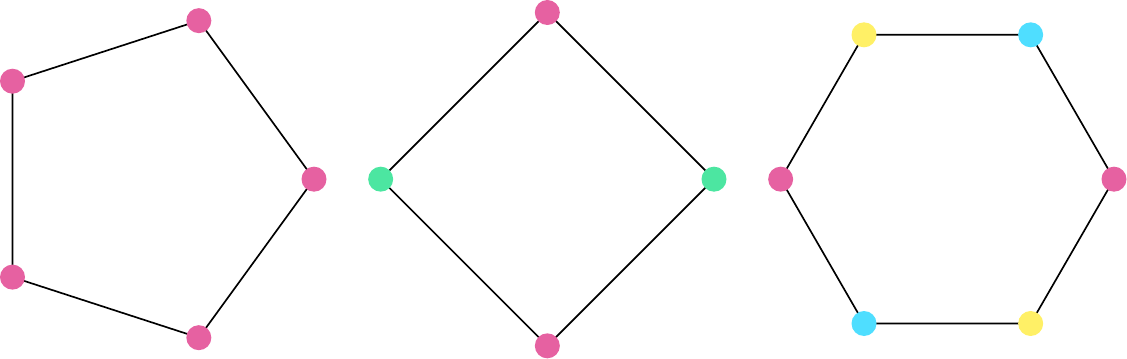}
   \caption{Optimal FAT colorings illustrating $\chi^{\mathrm{FAT}}(C_5)=1$, $\chi^{\mathrm{FAT}}(C_4)=2$, and $\chi^{\mathrm{FAT}}(C_6)=3$.}
   \label{fig:cycles}
\end{figure}

Now, the \emph{volume} of a set of vertices $S\subseteq V$ is defined as $$\vol S:=\sum_{v\in S}\deg v.$$

We conclude this section by proving that, in any FAT coloring of any connected graph, every coloring class has the same volume.

\begin{proposition}\label{prop:vol}
 Let $G=(V,E)$ be a graph, and let $k>1$. For any FAT $k$-coloring of $G$ with parameter $\alpha>0$, all coloring classes have the same volume. In particular, since $\vol V=2|E|$, each class has volume $2|E|/k$, and therefore $k$ divides $2|E|$.
\end{proposition}

\begin{proof}
For all $i\neq j$,
$$
e(V_i,V_j)=\sum_{v\in V_i} e(v,V_j)=\sum_{v\in V_i}\alpha\deg(v)=\alpha\vol(V_i).
$$
Symmetrically,
$$
e(V_i,V_j)=\sum_{u\in V_j} e(u,V_i)=\sum_{u\in V_j}\alpha\deg(u)=\alpha\vol(V_j).
$$
Hence, $\alpha\vol(V_i)=\alpha\vol (V_j)$. Since $\alpha>0$, we conclude that $\vol (V_i)=\vol(V_j)$ for all $i, j$.
\end{proof}

\begin{remark}
For general graphs, Proposition \ref{prop:vol} ensures equality of volumes of the coloring classes, but not of their sizes. In contrast, Theorem \ref{thm:regular} shows that in regular graphs every coloring class has the same size, therefore in that case the divisibility condition applies to $N$ as well as to $2|E|$.
\end{remark}

\section{Spectral aspects}\label{section:spectral}

Spectral graph theory provides a powerful framework for studying structural features of graphs through the eigenvalues and eigenvectors of associated operators, and many deep connections with graph coloring are known \cite{ElphickWocjan2015,Nikiforov2007,Gabriel-colouring,SunDas2020,beersmulas2024,Hoffman1970,generalizedHoffman}. In this section, we show that the existence of a FAT $k$-coloring imposes strong spectral constraints on the normalized Laplacian of a graph, and we derive bounds on the FAT chromatic number in terms of eigenvalue multiplicities. For regular graphs, we also translate these results to the adjacency matrix and the Kirchhoff Laplacian. \newline

We fix a simple graph $G=(V,E)$ on $N$ vertices.

\begin{definition}
The \emph{adjacency matrix} of $G$ is the $N\times N$ matrix $A:=A(G)$ with entries
$$
A_{ij} :=
\begin{cases}
1 & \text{if } v_i\sim v_j,\\
0 & \text{otherwise.}
\end{cases}
$$
The \emph{Kirchhoff Laplacian} of $G$ is
$$
K:=K(G):=D-A,
$$
where $D:=D(G)=\mathrm{diag}(\deg v_1,\ldots,\deg v_N)$ is the \emph{degree matrix} of $G$, and the \emph{normalized Laplacian} of $G$ is
$$
L:=L(G):=I-D^{-1}A.
$$
\end{definition}

The normalized Laplacian $L$ is known to have $N$ real non-negative eigenvalues, which are contained in $[0,2]$ and we denote by
$$
0=\lambda_1\leq \lambda_2 \leq \cdots \leq \lambda_N \leq 2.
$$

Such eigenvalues are known to encode several structural properties of the graph: for instance, the multiplicity of $\lambda_1=0$ equals the number of connected components of $G$, $\lambda_2$ approximates the Cheeger constant in the connected case, and $\lambda_N$ detects extremal structures, being equal to $2$ if and only if a connected component of $G$ is bipartite, and to $N/(N-1)$ if and only if $G$ is complete. \newline 

It is often convenient to view $L$ as a linear operator on the space $C(V)$ of functions $f:V\to\mathbb{R}$, given by
$$
Lf(v) = f(v) - \frac{1}{\deg v}\sum_{w\sim v} f(w), \quad v\in V,
$$
which is equivalent to the matrix formulation above. In particular, $f$ is an eigenfunction with eigenvalue $\lambda$ if and only if $Lf=\lambda f$, that is,
\begin{equation}\label{eq:eigenpair}
    \lambda f(v) = f(v) - \frac{1}{\deg v}\sum_{w\sim v} f(w), \quad v\in V.
\end{equation}

With the Courant-Fischer-Weyl min-max Principle below, we can characterize the eigenvalues of $L$.

\begin{theorem}[Courant-Fischer-Weyl min-max Principle]
\label{min-max theorem}

Let $H$ be an $N$-dimensional vector space with a positive definite scalar product $(.,.)$, and let $A:H\rightarrow H$ be a self-adjoint linear operator. Let $\mathcal{H}_{k}$  be the  family of all $k$-dimensional subspaces of $H$.
  Then the eigenvalues $\lambda_{1}\leq \ldots \leq \lambda_{N}$ of
$A$ can be obtained by 
\begin{equation}
\label{min-max}
\lambda_{k}=\min_{H_k\in \mathcal{H}_{k}}\max_{g(\neq 0)\in H_{k}}\frac{(Ag,g)}{(g,g)}=\max_{{H}_{N-k+1} \in \mathcal{H}_{N-k+1}}\min_{g(\neq 0)\in {H}_{N-k+1}}\frac{(Ag,g)}{(g,g)}.
\end{equation}
The vectors $g_k$ realizing such a min-max or max-min then are corresponding
eigenvectors, and the min-max spaces $H_{k}$ are spanned
by the eigenvectors for the eigenvalues $\lambda_1,\dots
,\lambda_k$, and analogously, the max-min spaces $H_{N-k+1}$ are spanned
by the eigenvectors for the eigenvalues $\lambda_k, \dots, \lambda_N$.\\

Thus, we also have

\begin{equation}
\label{min-max1}
\lambda_{k}=\min_{g(\neq 0)\in H, (g,g_j)=0\ \mathrm{ for }\ j=1,\dots ,k-1}\frac{(Ag,g)}{(g,g)}=\max_{g(\neq 0)\in H, (g,g_\ell)=0\ \mathrm{ for }\ \ell=k+1,\dots ,N}\frac{(Ag,g)}{(g,g)}.
\end{equation}

In particular, 
\begin{equation}
\label{min-max2}
\lambda_{1}=\min_{g(\neq 0)\in H}\frac{(Ag,g)}{(g,g)},\qquad \lambda_N=\max_{g(\neq 0)\in H}\frac{(Ag,g)}{(g,g)}.
\end{equation}
\end{theorem}
\begin{definition}\label{RQ}
$(Ag,g)/(g,g)$ is called the \emph{Rayleigh quotient} of $g$. 
\end{definition}

According to Theorem \ref{min-max theorem}, the eigenvalues of $L$ are given by min-max values of
\begin{equation*}
    \RQ(f):=\frac{\langle Lf,f\rangle}{\langle f,f\rangle}=\frac{\sum_{v\sim w}\biggl(f(v)-f(w)\biggr)^2}{\sum_{v\in V}\deg v\cdot f(v)^2}, \quad \text{for }f\in C(V).
\end{equation*}In particular, let $k\in \{1,\ldots,N\}$ and let $g_i$ be eigenfunctions for $\lambda_i$, for each $i\in \{1,\ldots,N\}\setminus \{k\}$, that are pairwise linearly independent. Then,
\begin{equation*}
    \lambda_k=\min_{\substack{f\in C(V)\setminus\{\mathbf{0}\}:\\ \langle f,g_1\rangle=\ldots=\langle f,g_{k-1}\rangle=0}} \RQ(f)=\max_{\substack{f\in C(V)\setminus\{\mathbf{0}\}:\\ \langle f,g_{k+1}\rangle=\ldots=\langle f,g_N\rangle=0}} \RQ(f),
\end{equation*}and the functions realizing such a min-max are the corresponding eigenfunctions for $\lambda_k$.

\begin{remark}\label{rmk:min-maxlambdaN}
The largest eigenvalue of the corresponding normalized Laplacian can be characterized by
\begin{equation*}
         \lambda_N=\max_{f:V\rightarrow \mathbb{R}} \frac{\sum_{v\sim w}\biggl(f(v)-f(w)\biggr)^2}{\sum_{v\in V}\deg v\cdot f(v)^2}.
\end{equation*}
Furthermore, any function $f\in C(V)\setminus \{\boldsymbol{0}\}$ attaining this maximum is an eigenfunction of $L$ with eigenvalue $\lambda_N$.
\end{remark}

We now introduce auxiliary functions associated with pairs of coloring classes. These are central in the proof of Theorem \ref{thm:lambda} below.

\begin{definition}\label{def:f_ij}
   Consider a vertex $k$-coloring of $G$ with coloring classes $V_1,\ldots,V_k$. For two distinct indices $i,j\in\{1,\ldots,k\}$, define the function $f_{ij}\colon V\to\R$ by
    \[
    f_{ij}(v) \coloneqq \begin{cases}
        1, &\text{ if } v\in V_i,\\
        -1, &\text{ if } v\in V_j,\\
        0, &\text{ otherwise.}
    \end{cases}
    \]
\end{definition}

\begin{theorem}\label{thm:lambda}
    If $G$ admits a FAT $k$-coloring with parameter $\alpha$, then $\lambda=k\alpha$ is an eigenvalue of $L(G)$ with multiplicity at least $\max\{1,k-1\}$.
\end{theorem}

\begin{proof}
Fix a FAT $k$-coloring with parameter $\alpha$, and let $V_1,\ldots,V_k$ be its coloring classes. \newline

By definition of $\beta$, we have 
$$
1-\beta+\alpha = 1-1+ (k-1)\alpha + \alpha = k\alpha = \lambda.
$$

If we can show that the functions $f_{ij}$ from Definition \ref{def:f_ij} are eigenfunctions of $L$ with eigenvalue $\lambda=1-\beta+\alpha$, then the result follows, since there are $k-1$ linearly independent such functions. In particular, we need to check that \eqref{eq:eigenpair} holds for all $v\in V$, with $f=f_{ij}$ and $\lambda=1-\beta+\alpha$.\newline

We fix the indices $i$ and $j$, and for $v\in V$ we consider three cases.

\begin{itemize}

\item Case 1: $v\in V_i$. In this case, $f_{ij}(v)=1$. Moreover, among the neighbors of $v$, a fraction $\beta$ lies in $V_i$, and a fraction $\alpha$ lies in $V_j$. Hence,
$$
f_{ij}(v)-\frac{1}{\deg v}\sum_{w\sim v}f_{ij}(w)=1-\frac{1}{\deg v}\left(\beta \deg v- \alpha \deg v \right)=1-\beta + \alpha = \lambda = \lambda \cdot f_{ij}(v).
$$
    
    \item Case 2: $v\in V_j$. In this case, $f_{ij}(v)=-1$. Moreover, among the neighbors of $v$, a fraction $\alpha$ lies in $V_i$, and a fraction $\beta$ lies in $V_j$. Hence,
$$
f_{ij}(v)-\frac{1}{\deg v}\sum_{w\sim v}f_{ij}(w)=-1-\frac{1}{\deg v}\left(\alpha \deg v- \beta \deg v \right)=-1+\beta - \alpha = -\lambda = \lambda \cdot f_{ij}(v).
$$

    \item Case 3: $v\notin V_i\cup V_j$. In this case, $f_{ij}(v)=0$. Moreover, among the neighbors of $v$, a fraction $\alpha$ lies in $V_i$, and a fraction $\alpha$ lies in $V_j$. Hence,
$$
f_{ij}(v)-\frac{1}{\deg v}\sum_{w\sim v}f_{ij}(w)=-\frac{1}{\deg v}\left(\alpha \deg v- \alpha \deg v \right)=0=\lambda \cdot f_{ij}(v).
$$

\end{itemize}
\end{proof}

\begin{remark} 
Proper $k$-colorings are FAT colorings with $\beta=0$. In this case, $\alpha=1/(k-1)$, and therefore Theorem \ref{thm:lambda} tells us that, for proper FAT colorings, $\lambda=k/(k-1)$ is an eigenvalue of $L$ with multiplicity at least $k-1$. This is also an immediate consequence of Proposition 3.5 in \cite{beersmulas2024}, and it generalizes two classical results:\\ 
(i) the fact that, for the complete graph $K_N$, $N/(N-1)$ is an eigenvalue of $L$ with multiplicity $N-1$ (by coloring each vertex with a different color), and\\
(ii) the fact that, for bipartite graphs, $\lambda_N=2$ (by taking any proper $2$-coloring).
\end{remark}

\begin{remark}
    A FAT $k$-coloring is a special case of a $k$-partition that is equitable with respect to $D^{-1}A$, as defined in \cite{Gabriel-colouring}. The concept of partitions that are equitable with respect to $D^{-1}A$ is the normalized Laplacian equivalent of the concept of partitions that are equitable (with respect to $A$), as defined in Section 9.3 of \cite{GodsilRoyle2001}. For regular graphs, the two concepts of equitability coincide.

    In \cite{beersmulas2024}, graphs for which the eigenvalue bound 
    \begin{equation}\label{eq:lambdan}
        \lambda_N \geq \frac{\chi}{\chi-1}
    \end{equation}
    is an equality are studied, where $\chi$ is the (proper) coloring number. In Proposition 3.5 of \cite{beersmulas2024}, it is shown that such graphs admit a proper FAT $\chi$-coloring avant la lettre. Consequently, FAT colorings generalize proper $\chi$-colorings of graphs for which \eqref{eq:lambdan} is an equality.

    For regular graphs, the bound \eqref{eq:lambdan} is equivalent to the Hoffman bound, an analogous bound concerning the smallest and largest eigenvalue of the adjacency matrix.
    A $\chi$-coloring of a graph for which the Hoffman bound is an equality is called a Hoffman coloring. Hoffman colorings for regular graphs, which are proper FAT colorings, are studied by Abiad, De Bruyn and Van Veluw in \cite{AbiadDeBruynVanVeluw2026, AbiadDeBruynVanVeluw2026.2}.

\end{remark}

The next result is the analogue of Proposition \ref{prop:delta+1}, with the maximum eigenvalue multiplicity of $L(G)$ in place of the minimum vertex degree of $G$.

\begin{corollary}\label{cor:lambda}
    Let $\mu$ be the maximum eigenvalue multiplicity of $L(G)$. If $G$ admits a FAT $k$-coloring, then $k\leq \mu+1$. In particular, $$\chi^{\textrm{FAT}}\leq \mu+1.$$ Moreover, if $G=K_N$, then the bound is sharp.
\end{corollary}

\begin{proof}
    By Theorem \ref{thm:lambda}, if $G$ admits a FAT $k$-coloring, then $\mu \geq k-1$, hence $k\leq \mu+1$. For the complete graph $K_N$, we have that $\chi^{\textrm{FAT}}=N$ while $\mu=N-1$, therefore equality holds.
\end{proof}

For regular graphs, we can translate Theorem \ref{thm:lambda} in terms of the adjacency matrix and the Kirchhoff Laplacian, as follows.

\begin{theorem}\label{thm:spectrumAK}
    Let $G$ be a $d$-regular graph that admits a FAT $k$-coloring with parameter $\alpha$. Then, $d k\alpha$ is an eigenvalue of $K(G)$ with multiplicity at least $\max\{1,k-1\}$. Similarly, $d(1-k\alpha)$ is an eigenvalue of $A(G)$ with multiplicity at least $\max\{1,k-1\}$.
\end{theorem}

\begin{proof}
For a $d$-regular graph, $K=d\cdot L$ and $A=d\cdot (\id-L)$. Therefore, in this case,
\begin{align*}
    \lambda \text{ is an eigenvalue for }L
    &\iff d \lambda \text{ is an eigenvalue for }K\\
    &\iff d(1-\lambda) \text{ is an eigenvalue for }A,
\end{align*} with the same multiplicities. In particular,
\begin{align*}
    k\alpha \text{ is an eigenvalue for }L
    &\iff d k\alpha \text{ is an eigenvalue for }K\\
    &\iff d(1-k\alpha) \text{ is an eigenvalue for }A,
\end{align*}with the same multiplicities. Together with Theorem \ref{thm:lambda}, this proves the claim.
\end{proof}

We conclude this section by generalizing the chromatic bound from \eqref{eq:lambdan}, and showing that the equality case is closely related to FAT colorings.

First, we generalize the definition of the proper coloring number as follows.
\begin{definition}\label{def:betatolerant}
    Let $\beta\in[0,1]$.
    A vertex $k$-coloring with non-empty classes $V_1,\ldots,V_k$ is \emph{$\beta$-tolerant} if, for every $i=1,\ldots,k$ and for every $v\in V_i$, 
    \begin{equation*}
        e(v,V_i) \leq \beta \cdot \deg v.
    \end{equation*}
    The $\beta$-tolerant coloring number $\chi_{\mathrm{tol}}^\beta = \chi_{\mathrm{tol}}^\beta(G)$ is the minimum number of colors needed to color $G$ $\beta$-tolerantly.
\end{definition}
\begin{remark}
    The $\beta$-tolerant coloring number $\chi_{\mathrm{tol}}^0$ generalizes $\chi$, the (proper) chromatic number, since 
    $\chi_{\mathrm{tol}}^0 = \chi$. It also generalizes the \emph{majority chromatic number}, which coincides with $\chi_{\mathrm{tol}}^{1/2}$. It is well-known that, for any simple graph $G$, $\chi_{\mathrm{tol}}^{1/2}(G) \leq 2$  \cite{bosek2019majority}. 
    Hence, for any $\beta\in [1/2,1]$, it holds that $\chi_{\mathrm{tol}}^{\beta}(G) \leq 2$.

    Since for any $\beta\in [0,1)$ and connected graph $G$ on at least $2$ vertices, we have $\chi_{\mathrm{tol}}^{\beta}(G)\geq2$, it must hold that $\chi_{\mathrm{tol}}^{\beta}(G) = 2$ for $\beta\in[1/2,1)$. Furthermore, $\chi_{\mathrm{tol}}^{1}(G) = 1$.
\end{remark}

\begin{remark}
    The notion of $\beta$-tolerant colorings is closely related to that of \emph{$d$-improper colorings}, i.e., colorings in which every monochromatic component has degree at most $d$. The corresponding \emph{$d$-improper coloring number} is denoted $\chi^d(G)$. The bound in the theorem below, that concerns $\lambda_N$ and $\chi_{\mathrm{tol}}^{\beta}(G)$, is closely related to a similar bound, concerning the smallest and largest eigenvalue of the adjacency matrix,  $\mu_N$ and $\mu_1$, respectively:
    \begin{equation}\label{eq:Bilubound}
        \chi^d(G) \geq \frac{\mu_1-\mu_N}{d-\mu_N}.
    \end{equation}
    This generalization of the Hoffman bound was first proven by Bilu in \cite{Bilu2006}, and the equality case was studied by Guo, Kang and Zwaneveld in Theorem 10 of \cite{GuoKangZwaneveld2024}. For regular graphs, the bound in \eqref{eq:Bilubound} and in the theorem below coincide.

    In \cite{GuoKangZwaneveld2024}, a \emph{$d$-improper Hoffman coloring} is defined as a $\chi^d$-coloring of a graph attaining the $d$-improper Hoffman bound \eqref{eq:Bilubound}. We shall show in the theorem below that a special class of FAT colorings, i.e., FAT $k$-colorings with parameters $\alpha$ and $\beta$, such that $\lambda_N = k\alpha$, are the normalized Laplacian equivalent of $d$-improper Hoffman colorings.
\end{remark}

The following theorem generalizes \eqref{eq:lambdan} as follows.
\begin{theorem}
    Let $G$ be a graph without isolated vertices and let $\beta\in[0,1)$. Then,
    \begin{equation}\label{eq:chitolbeta}
        \lambda_N \geq (1-\beta) \cdot \frac{\chi_{\mathrm{tol}}^{\beta}}{\chi_{\mathrm{tol}}^{\beta}-1}.
    \end{equation}
    Furthermore, if the bound from \eqref{eq:chitolbeta} is an equality, then $G$ admits a FAT $\chi_{\mathrm{tol}}^{\beta}$-coloring with
    \[\alpha=\frac{1-\beta}{\chi_{\mathrm{tol}}^{\beta}-1},\] and the multiplicity of the eigenvalue $(1-\beta)\cdot \chi_{\mathrm{tol}}^{\beta}/\bigl(\chi_{\mathrm{tol}}^{\beta}-1\bigr)$ is at least $\chi_{\mathrm{tol}}^{\beta}-1$.
\end{theorem}
\begin{proof}
    To prove the inequality, let $V_1,\ldots,V_{\chi_{\mathrm{tol}}^{\beta}}$ be the coloring classes of a $\beta$-tolerant $\chi_{\mathrm{tol}}^{\beta}$-coloring. Furthermore, assume without loss of generality that
    \begin{equation}\label{eq:assumption}
        e(V_1,V_2) = \max_{1\leq i<j\leq \chi_{\mathrm{tol}}^{\beta}} e(V_i,V_j).
    \end{equation}
    Now let $f_{12}$ be as in Definition \ref{def:f_ij}. Then, by Theorem \ref{min-max theorem},
    \begin{align}
        \nonumber \lambda_N &\geq \RQ\bigl(f_{12}\bigr)\\
        \nonumber &= \frac{\sum_{v\sim w}\bigl(f_{12}(v)-f_{12}(w)\bigr)^2}{\sum_{v\in V}\deg v \cdot f_{12}(v)^2}\\
        \nonumber &= 1-2 \frac{\sum_{v\sim w}f_{12}(v) \cdot f_{12}(w)}{\vol V_1 + \vol V_2}\\
        \nonumber &= 1-\frac2{\vol V_1 + \vol V_2}\left(\sum_{\substack{v\sim w\\ v\in V_1, w\in V_2}} -1 + \sum_{\substack{v\sim w\\ v,w\in V_1,}} 1 + \sum_{\substack{v\sim w\\ v,w\in V_2}} 1 \right)\\
        \nonumber &= 1 + \frac{2e\bigl(V_1,V_2\bigr)}{\vol V_1+ \vol V_2} -\frac{2e\bigl(V_1,V_1\bigr) + 2e\bigl(V_2,V_2\bigr)}{\vol V_1 + \vol V_2} \\
        \label{eq:line1} &\geq 1 + \frac1{\chi_{\mathrm{tol}}^{\beta}-1}\cdot \frac{\vol V_1 - 2e\bigl(V_1,V_1\bigr) + \vol V_2 - 2e\bigl(V_2,V_2\bigr)}{\vol V_1 + \vol V_2} -\frac{2e\bigl(V_1,V_1\bigr) + 2e\bigl(V_2,V_2\bigr)}{\vol V_1 + \vol V_2}\\
        \nonumber &= 1 + \frac1{\chi_{\mathrm{tol}}^{\beta}-1} -  \Bigl(1 + \frac1{\chi_{\mathrm{tol}}^{\beta}-1}\Bigr)\cdot \frac{2e\bigl(V_1,V_1\bigr) + 2e\bigl(V_2,V_2\bigr)}{\vol V_1 + \vol V_2}\\
        \label{eq:line2} &\geq 1 + \frac1{\chi_{\mathrm{tol}}^{\beta}-1} - \Bigl(1 + \frac1{\chi_{\mathrm{tol}}^{\beta}-1}\Bigr)\cdot \beta\\
        \nonumber &= (1-\beta) \cdot \frac{\chi_{\mathrm{tol}}^{\beta}}{\chi_{\mathrm{tol}}^{\beta}-1}.
    \end{align}
    In \eqref{eq:line1}, we use that, by assumption \eqref{eq:assumption},
    \begin{align*}
        \vol V_1 - 2e\bigl(V_1,V_1\bigr) &= \sum_{i=2}^{\chi_{\mathrm{tol}}^{\beta}} e \bigl( V_1,V_i\bigr) \\
        &\leq \bigl(\chi_{\mathrm{tol}}^{\beta} - 1 \bigr) \cdot e\bigl(V_1,V_2\bigr),
    \end{align*}
    and, analogously, \[\vol V_2 - 2e\bigl(V_2,V_2\bigr) \leq \bigl(\chi_{\mathrm{tol}}^{\beta} - 1 \bigr) \cdot e\bigl(V_1,V_2\bigr).\]
    Furthermore, in \eqref{eq:line2}, we use the definition of $\chi_{\mathrm{tol}}^{\beta}$ to find that, for $i=1,2$, we have
    \begin{align*}
        \sum_{v\in V_i}e\bigl(v,V_i\bigr) &\leq \sum_{v\in V_i}\beta \cdot \deg v\\
        2e\bigl(V_i,V_i\bigr) &\leq \beta \cdot \vol V_i.
    \end{align*}
    That concludes the proof of the inequality.\\

    Moreover, if the inequality from \eqref{eq:chitolbeta} is an equality, then $f_{12}$ is an eigenfunction with eigenvalue $\lambda_N = (1-\beta)\cdot \chi_{\mathrm{tol}}^{\beta}/\bigl(\chi_{\mathrm{tol}}^{\beta}-1\bigr)$ by Remark \ref{rmk:min-maxlambdaN}. Furthermore, since the inequality from \ref{eq:line1} is also an equality, it follows that
    \[
    e\bigl(V_i,V_j\bigr)=e\bigl(V_1,V_2\bigr)\quad \text{for all} \quad 1\leq i<j\leq \chi_{\mathrm{tol}}^{\beta}.
    \]
    Then, following the proof for $f_{12}$, for all $1\leq i<j\leq \chi_{\mathrm{tol}}^{\beta}$, we have $\lambda_N = \RQ(f_{ij})$, and, consequently, $f_{ij}$ is an eigenfunction with eigenvalue $\lambda_N$ (again by Remark \ref{rmk:min-maxlambdaN}).

    Now, let
    \[
    \alpha \coloneqq \frac{1-\beta}{\chi_{\mathrm{tol}}^{\beta}-1} \quad \text{and} \quad k \coloneqq \chi_{\mathrm{tol}}^{\beta}.
    \]
    Let $i=1,\ldots,k$ and let $v\in V_i$. Then, for any $j=1,\ldots,k$ that is distinct from $i$, we have, by \eqref{eq:eigenpair},
    \begin{align*}
        \frac1{\deg v} \sum_{w\sim v} f_{ij}(w) &= (1 - k\alpha)f_{ij}(v)\\
        \frac{e\bigl(v,V_i\bigr) - e\bigl(v,V_j\bigr)}{\deg v} &= 1-k\alpha.
    \end{align*}
    It follows that $e\bigl(v,V_j\bigr)$ does not depend on $j$ for $j\neq i$. Now, we have that
    \begin{align*}
        \sum_{j\neq i} \frac{e\bigl(v,V_i\bigr) - e\bigl(v,V_j\bigr)}{\deg v} &= (1-k\alpha)(k-1)\\
        (k-1)\cdot \frac{e\bigl(v,V_i\bigr)}{\deg v} - \frac{\deg v - e\bigl(v,V_i\bigr)}{\deg v} &= (1-k\alpha)(k-1)\\
        k\cdot \frac{e\bigl(v,V_i\bigr)}{\deg v} &= k\bigl(1-\alpha(k-1)\bigr)\\
        \frac{e\bigl(v,V_i\bigr)}{\deg v} &= \beta.
    \end{align*}
    Since $e\bigl(v,V_j\bigr)$ does not depend on $j$ for $j\neq i$, we have
    \[
    \frac{e\bigl(v,V_i\bigr)}{\deg v} = \alpha,
    \]
    and we conclude that the coloring with coloring classes $V_1,\ldots,V_k$ is a FAT $k$-coloring with parameters $\alpha$ and $\beta$.

    The final part of the statement follows from Theorem \ref{thm:lambda}.
\end{proof}

\section{Regular Turán graphs}\label{section:Turan}

In this section, we characterize the FAT chromatic number of regular Turán graphs. Before stating the main result, we recall the definitions of complete multipartite graphs, Turán graphs and their regular case, and we fix the notations that will be used throughout this section.

 \begin{figure}[h]
    \centering
\includegraphics[scale=0.66]{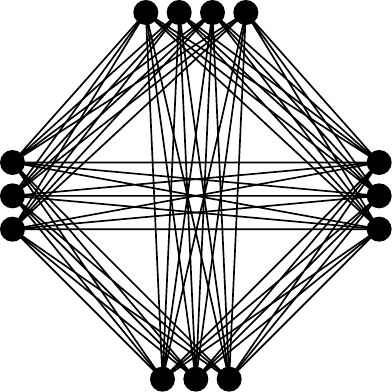}
    \caption{The Turán graph $T(13,4)$.}
    \label{fig:Turan}
\end{figure}

\begin{definition}\label{def:parts}
    A \emph{complete multipartite graph} is a graph whose vertex set is partitioned into disjoint independent sets $P_1,\dots,P_t$ (called \emph{parts}), with an edge between every two vertices in different parts. If the parts have sizes $n_1,\dots,n_t$, we denote the graph by $K_{n_1,\dots,n_t}$.\end{definition}

    \begin{definition}
            The \emph{Turán graph} $T(N,t)$ (Figure \ref{fig:Turan}) is the complete $t$-partite graph on $N$ vertices whose parts are as equal in size as possible, that is, each part has either $\lfloor N/t \rfloor$ or $\lceil N/t \rceil$ vertices. In this case, the parts $P_1,\dots,P_t$ are also called \emph{Turán parts}.
    \end{definition}

Turán graphs play a central role in extremal combinatorics \cite{turan1941,erdHos1981combinatorial,katona1964problem}, and here they provide a natural class of graphs for testing the behavior of the FAT chromatic number. 

\begin{definition}\label{def:Turangraph}
 A \emph{regular Turán graph}  is a Turán graph $T(N,t)$ in which $t$ divides $N$, so that all parts have the same size $N/t$. In this case, the graph is $d$-regular, with degree 
      $$d = N - N/t.$$
\end{definition}

We are now ready to state the main result of this section.

\begin{theorem}\label{thm:Turan}
 Let $T(N,t)$ be the regular Turán graph with $N$ vertices and $t$ equal parts $P_1,\ldots,P_t$ of size $N/t$. Every FAT coloring of $T(N,t)$ falls into one of the following two cases (Figure \ref{fig:T(12,4)bm}):
\begin{enumerate}
    \item \emph{Balanced case.} Each Turán part contains all colors, in equal sizes.
    \item \emph{Monochromatic case.} Each Turán part contains exactly one color, and each color appears on exactly the same number of Turán parts. 
\end{enumerate}
    As a consequence, $$\chi^{\textrm{FAT}}\bigl(T(N,t)\bigr)=\max\bigl\{t,N/t\bigr\}.$$
\end{theorem}

\begin{figure}[h]
    \centering
    \includegraphics[scale=0.66]{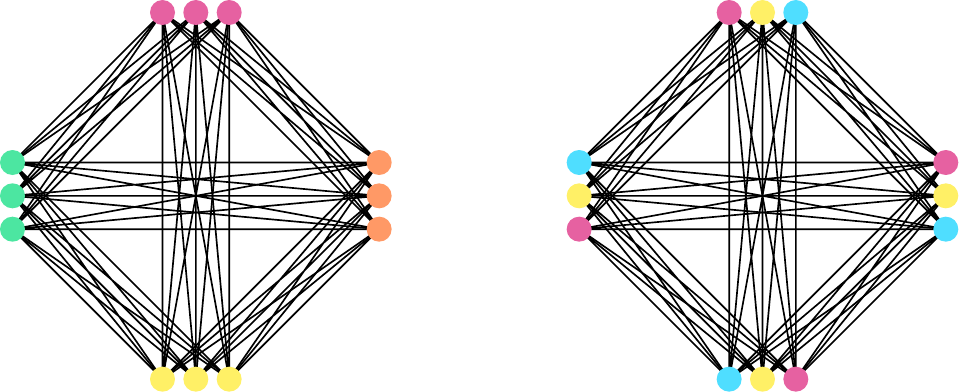}
    \caption{The Turán graph $T(12,4)$ with a monochromatic (left) and balanced (right) FAT coloring.}
    \label{fig:T(12,4)bm}
\end{figure}

We split the proof of Theorem \ref{thm:Turan} into four parts (three propositions and one corollary). The first claim tells us that, given a FAT coloring and a Turán part $P_g$ of $T(N,t)$, all colors that appear in $P_g$ occur with the same multiplicity. 

\begin{proposition}\label{claim1:turan}
    Let $T(N,t)$ be the regular Turán graph with $N$ vertices and $t$ equal parts $P_1,\ldots,P_t$ of size $N/t$. Fix a FAT $k$-coloring with coloring classes $V_1,\ldots,V_k$, and let $g\in\{1,\ldots,t\}$. If two distinct coloring classes $V_i$ and $V_j$ are such that $V_i\cap P_g \neq \emptyset$ and $V_j\cap P_g \neq \emptyset$, then 
    $$
    |V_i \cap P_g|=|V_j \cap P_g|.
    $$ 
\end{proposition}

\begin{proof}

Let $d:=N-N/t$ denote the degree of all vertices in $T(N,t)$. Fix $v_i\in V_i\cap P_g$, and $v_j\in V_j\cap P_g$. Then, 

$$
    e(v_j,V_i)=|V_i|-|V_i \cap P_g|=\alpha d,
    $$ 
   using the fact that $v_j\in P_g$ and $v_j\notin V_i$, and similarly 
    $$
    e(v_i,V_j)=|V_j|-|V_j \cap P_g|=\alpha d,
    $$ 
    using the fact that $v_i\in P_g$ and $v_i\notin V_j$.
    Therefore,
    $$
    \alpha d=|V_i|-|V_i \cap P_g|=|V_j|-|V_j \cap P_g|,
    $$ implying that 
    $$
 |V_i \cap P_g|-|V_j \cap P_g|  = |V_i|- |V_j|  =0,
    $$ where the last equality follows by Theorem \ref{thm:regular}. Hence, $|V_i \cap P_g|=|V_j \cap P_g|$.
\end{proof}

The next claim tells us that, given a FAT coloring and a Turán part $P_g$ of $T(N,t)$, if $P_g$ contains at least two colors, then it contains all colors.

\begin{proposition}\label{claim2:turan}
    Let $T(N,t)$ be the regular Turán graph with $N$ vertices and $t$ equal parts $P_1,\ldots,P_t$ of size $N/t$. Fix a FAT $k$-coloring with coloring classes $V_1,\ldots,V_k$, and let $g\in\{1,\ldots,t\}$. If two distinct coloring classes $V_i$ and $V_j$ are such that $V_i\cap P_g \neq \emptyset$ and $V_j\cap P_g \neq \emptyset$, then for all $z\neq i,j$, 
    $$
    V_z\cap P_g \neq \emptyset.
    $$ 
\end{proposition}

\begin{proof}

Let $d:=N-N/t$ denote the degree of all vertices in $T(N,t)$, and fix $v_i\in V_i\cap P_g$. Then, 

$$
    e(v_i,V_j)=|V_j|-|V_j \cap P_g|=\alpha d,
    $$ using the fact that $v_i\in P_g$ and $v_i\notin V_j$, and similarly
    $$
    e(v_i,V_z)=|V_z|-|V_z \cap P_g|=\alpha d,
    $$ using the fact that $v_i\in P_g$ and $v_i\notin V_z$.   Therefore,
    $$
    \alpha d=|V_j|-|V_j \cap P_g|=|V_z|-|V_z \cap P_g|,
    $$ implying that 
    $$
 |V_j \cap P_g|-|V_z \cap P_g|  = |V_j|- |V_z|  =0,
    $$ where the last equality follows by Theorem \ref{thm:regular}. Hence, $|V_j \cap P_g|=|V_z \cap P_g|$. Since $V_j\cap P_g \neq \emptyset$, this implies that $V_z\cap P_g \neq \emptyset$.
\end{proof}

The third claim tells us that, given a FAT coloring, a Turán part is monochromatic if and only if all other parts are monochromatic.

\begin{proposition}\label{claim3:turan}
Let $T(N,t)$ be the regular Turán graph with $N$ vertices and $t$ equal parts $P_1,\ldots,P_t$ of size $N/t$. Given a FAT coloring and two distinct Turán parts $P_g$ and $P_s$,
$$
P_g \text{ is monochromatic} \iff  P_s \text{ is monochromatic.}
$$
\end{proposition}

\begin{proof}Let $V_1,\ldots,V_k$ be the coloring classes of the given FAT coloring, and let $d:=N-N/t$ denote the degree of all vertices in $T(N,t)$.\newline 

Assume, by contradiction, that $P_g\subseteq V_i$ (that is, $P_g$ contains only color $i$), while $P_s$ contains at least two distinct colors. Then, by Proposition \ref{claim2:turan}, $P_s$ must contain all colors. In particular, $V_i \cap P_s \neq \emptyset$, and therefore we can fix $w_i \in V_i \cap P_s$. Moreover, given $j\neq i$, we must have that $V_j \cap P_g = \emptyset$, while $V_j \cap P_s \neq \emptyset$. Fix also $v_i\in P_g$. Then, $v_i\in V_i$.\newline 

We have that

$$ e(v_i,V_j)=|V_j|=\alpha d, $$ 

using the fact that $v_i\in P_g$, $V_j \cap P_g = \emptyset$ and $v_i\notin V_j$, and similarly

$$ e(w_i,V_j)=|V_j|-|V_j\cap P_s|=\alpha d, $$     using the fact that $w_i\in P_s$ and $w_i\notin V_j$. This implies that 

$$
\alpha d=|V_j|=|V_j|-|V_j\cap P_s|,
$$ therefore $|V_j\cap P_s|=0$. This is a contradiction, since $V_j \cap P_s \neq \emptyset$.

\end{proof}

We are now ready to prove Theorem \ref{thm:Turan}.

\begin{corollary*}[Theorem \ref{thm:Turan}]
   Let $T(N,t)$ be the regular Turán graph with $N$ vertices and $t$ equal parts $P_1,\ldots,P_t$ of size $N/t$. Every FAT coloring of $T(N,t)$ falls into one of the following two cases (Figure \ref{fig:T(12,4)bm}):
\begin{enumerate}
    \item \emph{Balanced case.} Each Turán part contains all colors, in equal sizes.
    \item \emph{Monochromatic case.} Each Turán part contains exactly one color, and each color appears on exactly the same number of Turán parts. 
\end{enumerate}
    As a consequence, $$\chi^{\textrm{FAT}}(T(N,t))=\max\{t,N/t\}.$$
\end{corollary*}

\begin{proof} The first claim follows immediately from Propositions \ref{claim1:turan}, \ref{claim2:turan}, and \ref{claim3:turan}, together with Theorem \ref{thm:regular}. To show that $$\chi^{\textrm{FAT}}(T(N,t))=\max\{t,N/t\},$$ observe that:

\begin{enumerate}
    \item In the \emph{balanced case}, the maximum possible number of colors is $|P_g|=N/t$.
    \item In the \emph{monochromatic case}, the maximum possible number of colors is $t$.
\end{enumerate}
\end{proof}

\section{Irreducible FAT colorings}\label{section:irr}
We now shine a light on the question of determining for which values of $\alpha$ a graph $G$ admits a FAT coloring. A necessary condition can be derived from Theorem \ref{thm:lambda}: if $k\alpha$ is not an eigenvalue of the normalized Laplacian $L(G)$ with multiplicity at least $k-1$ for any integer $k$, then $\alpha$ cannot be the parameter of a FAT coloring of $G$.

Finding all feasible values of $\alpha$ directly appears difficult, since it apparently requires describing all FAT colorings of $G$. In this section, we study relationships between different FAT colorings and show that some can be obtained from others by merging coloring classes. This allows us to reduce the determination problem to finding a smaller set of non-decomposable FAT colorings, called \emph{irreducible} FAT colorings, from which all others can be derived.

We start with an illustrative example.
In Theorem \ref{thm:Turan}, we determined the FAT coloring number of regular Turán graphs by characterizing all of their FAT colorings. Two types of FAT Turán colorings arose: balanced and monochromatic colorings. Moreover, we saw that each coloring class of a monochromatic FAT coloring of a regular Turán graph $T(N,t)$  is a union of $\ell$ Turán parts, for some fixed $\ell\mid t$. In other words, for any $\ell\mid t$, one obtains a monochromatic FAT $\ell$-coloring by merging the $t$ Turán parts $P_1,\ldots,P_t$, which are also the coloring classes of the unique proper $t$-coloring of $T(N,t)$, into $\ell$ groups of size $t/\ell$. Similarly, in this case, balanced FAT colorings can be constructed as follows. Consider a FAT $N/t$-coloring whose coloring classes $V_1,\ldots,V_{N/t}$ each contain exactly one vertex from every Turán part. By Theorem \ref{thm:Turan}, for any divisor $m$ of $N/t$, one obtains a balanced FAT $m$-coloring by taking unions of $(N/t)/m$ coloring classes of a balanced $N/t$-coloring.

These observations suggest an underlying principle: FAT colorings with fewer coloring classes can often be obtained by taking unions of coloring classes of FAT colorings with more coloring classes (Figure \ref{fig:merging}).

\begin{figure}[h]
    \centering
    \includegraphics[scale=0.6]{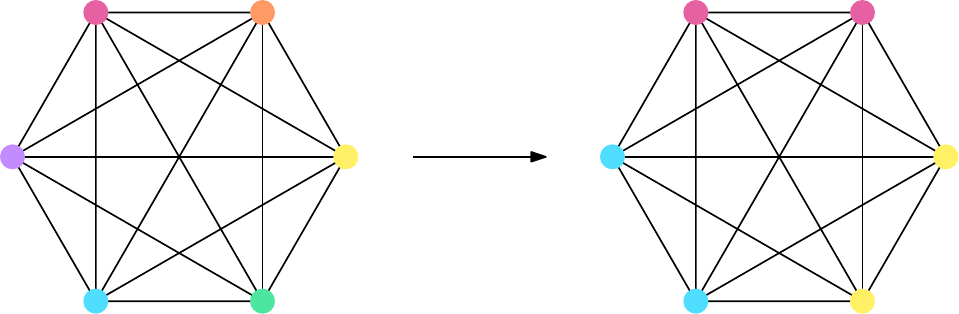}
    \caption{Obtaining a new FAT coloring by merging coloring classes.}
    \label{fig:merging}
\end{figure}

\begin{theorem}\label{thm:merging}
    Let $c_1$ be a FAT $k$-coloring of $G$ with coloring classes
    $V_1^{(1)},\ldots,V_k^{(1)}$ 
    and parameters $\alpha^{(1)}$ and $\beta^{(1)}$.
    For any $\ell\mid k$ with $\ell>1$, one can construct a FAT $\ell$-coloring $c_2$ with parameters
    \begin{equation*}
        \alpha^{(2)} = \frac k\ell\cdot \alpha^{(1)} \quad\text{and} \quad \beta^{(2)} = \beta^{(1)} + \biggl(\frac k\ell-1\biggr)\alpha^{(1)}
    \end{equation*}
    by taking unions of $k/\ell$ coloring classes of $c_1$ to form $\ell$ coloring classes of $c_2$.
\end{theorem}
\begin{proof}
    Let $\ell\mid k$, and define the partition $V_1^{(2)},\ldots,V_\ell^{(2)}$ by letting $V_i^{(2)}$ be the union of the $i$-th block of $k/\ell$ coloring classes of $c_1$, i.e.,
    \[
    V_i^{(2)} \coloneqq \bigcup_{j=1}^{k/\ell} V_{(i-1)\cdot k/\ell+j}^{(1)}.
    \]
    We claim that the coloring $c_2$ with coloring classes $V_i^{(2)}$ is a FAT $\ell$-coloring with the parameters $\alpha^{(2)}$ and $\beta^{(2)}$ given in the statement.
  To see this, let $v\in V(G)$, and fix $m$ such that $v\in V^{(2)}_m$. Moreover, consider any $n\in \bigl\{1,\ldots,\ell\bigr\}$ with $n\neq m$. Since $v$ has $\alpha^{(1)}\cdot\deg v$ neighbors in each of the coloring classes $V^{(1)}_{(n-1)\cdot k/\ell + j}$ of $c_1$, and since $V^{(2)}_n$ is a union of $k/\ell$ of those classes, $v$ has $$\frac k\ell\cdot \alpha^{(1)} \cdot \deg v $$ neighbors in $V_n^{(2)}$.
    Hence, $c_2$ is a FAT $\ell$-coloring with parameter $\alpha^{(2)}$.
    
    Finally, we have that
    \begin{align*}
        \beta^{(2)} &= 1-(\ell-1)\cdot \alpha^{(2)} \\
        &= 1-(\ell-1)\cdot \frac k\ell\cdot \alpha^{(1)}\\
        &= 1-(k-1)\cdot \alpha^{(1)} + \biggl(\frac k\ell-1\biggr)\cdot \alpha^{(1)}  \\
        &=\beta^{(1)} + \biggl(\frac k\ell-1\biggr)\alpha^{(1)}.
    \end{align*}
    This proves the claim.
\end{proof}
We now formalize the idea of obtaining new FAT colorings from existing ones by taking unions of coloring classes.
This construction naturally induces a partial order on the set of all FAT colorings of $G$.
\begin{definition}\label{def:po}
    Let $c_1$ and $c_2$ be FAT colorings of $G$. 
    We say that $c_1$ is \emph{finer} than $c_2$, and that $c_2$ is \emph{coarser} than $c_1$, denoted $c_2\preceq c_1$, if every coloring class of $c_2$ is a union of some coloring classes of $c_1$.
\end{definition}

\begin{remark}
    A straightforward verification shows that the relation introduced in Definition \ref{def:po} defines a partial order on the set of all FAT colorings of $G$.    
\end{remark}

Now, every FAT coloring is either a maximal element of this partial order or is coarser than some maximal element of the partial order.
Therefore, to describe all FAT colorings of $G$, it suffices to find its maximal elements.
This motivates the following terminology.
\begin{definition}
    A FAT coloring that is maximal with respect to the partial order from Definition \ref{def:po} is called an \emph{irreducible} FAT coloring.
    Any FAT coloring that is not irreducible is called \emph{reducible}.
\end{definition}

\begin{example}
    Any FAT coloring of $G$ is finer than the unique FAT $1$-coloring. Consequently, the FAT $1$-coloring is irreducible if and only if $\chi^{\mathrm{FAT}}=1$.
\end{example}

The next theorem shows that any reducible FAT $\ell$-coloring $c_2$ can be obtained from an irreducible FAT $k$-coloring $c_1$, by taking unions of $k/\ell$ coloring classes of $c_1$ to form $\ell$ coloring classes of $c_2$. In particular, $k$ must be a multiple of $\ell$.
\begin{theorem}\label{thm:gnigrem}
    Let $G$ be a connected graph, and
    let $c_1$ and $c_2$ be FAT $k$- and $\ell$-colorings of $G$, respectively. If $c_2\preceq c_1$, then every coloring class of $c_2$ is a union of $k/\ell$ coloring classes of $c_1$.
\end{theorem}
\begin{proof}
    First, note that if $\ell=1$, then the statement follows. Therefore, we henceforth assume that $\ell > 1$.

    Let $\alpha^{(1)}$ and $\beta^{(1)}$ be the FAT parameters of $c_1$, and let $\alpha^{(2)}$ and $\beta^{(2)}$ be the FAT parameters of $c_2$. We note that $\beta^{(2)}<1$, because $G$ is connected, and because we assumed that $\ell > 1$. Consequently, $\beta^{(1)}<1$, and, in particular, $\alpha^{(1)}>0$.
    
    Furthermore, let $V_1^{(2)}$ and $V_2^{(2)}$ be coloring classes of $c_2$. Since $c_2 \preceq c_1$, every coloring class of $c_2$ is a union of coloring classes of $c_1$. Let $m_i$ be such that $V_i^{(2)}$ is the union of $m_i$ coloring classes of $c_1$, for $i=1,2$.
    
    Note that any vertex $v_i \in V_i^{(2)}$ has $\beta^{(1)}\deg v_i$ neigbors in its coloring class from $c_1$, that we denote by $V_i^{(1)}$, and $\alpha^{(1)}\deg v_i$ neighbors in the $m_i-1$ coloring classes of $c_1$ that were merged with $V_i^{(1)}$ to form $V_i^{(2)}$. Therefore,
    we have
    \begin{align*}
        \beta^{(2)} &= \beta^{(1)} + (m_1-1) \cdot \alpha^{(1)}\\
        &= \beta^{(1)} + (m_2 - 1) \cdot \alpha^{(1)},
    \end{align*}
    and since $\alpha^{(1)}>0$, it follows that $m_1=m_2$.

    Hence, by induction, every coloring class of $c_2$ equals the union of $m_1$ coloring classes of $c_1$, and consequently, $m_1=k/\ell$.
\end{proof}

The next result shows that proper FAT colorings always correspond to maximal elements of the partial order.
\begin{proposition}
    Any proper FAT coloring of a connected graph $G$ is irreducible.
\end{proposition}
\begin{proof}
   We prove the claim by showing that a reducible FAT coloring cannot be proper.
   
    Let $c_2$ be a reducible FAT $\ell$-coloring of $G$. Then, there exists a FAT $k$-coloring $c_1$ of $G$, with $k>\ell$, such that $c_2\preceq c_1$.
    By Theorem \ref{thm:gnigrem} and because $G$ is connected, the coloring classes of $c_2$ each are a union of $k/\ell$ coloring classes of $c_1$.
    By Theorem \ref{thm:merging}, we have
    \[
    \beta^{(2)} = \beta^{(1)} + \biggl(\frac k\ell - 1 \biggr)\alpha^{(1)},
    \]
    where $\alpha^{(1)}>0$ and $k/\ell >1$. Hence $\beta^{(2)}>0$, and  therefore $c_2$ cannot be proper.
    We conclude that any proper FAT coloring of a connected graph must be irreducible.
\end{proof}

    We now return to the motivating question of this section: for which values of $\alpha$ does the regular Turán graph $T(N,t)$ admit a FAT coloring with parameter $\alpha$?

\begin{figure}[h!]
    \centering
    \includegraphics[scale=0.66]{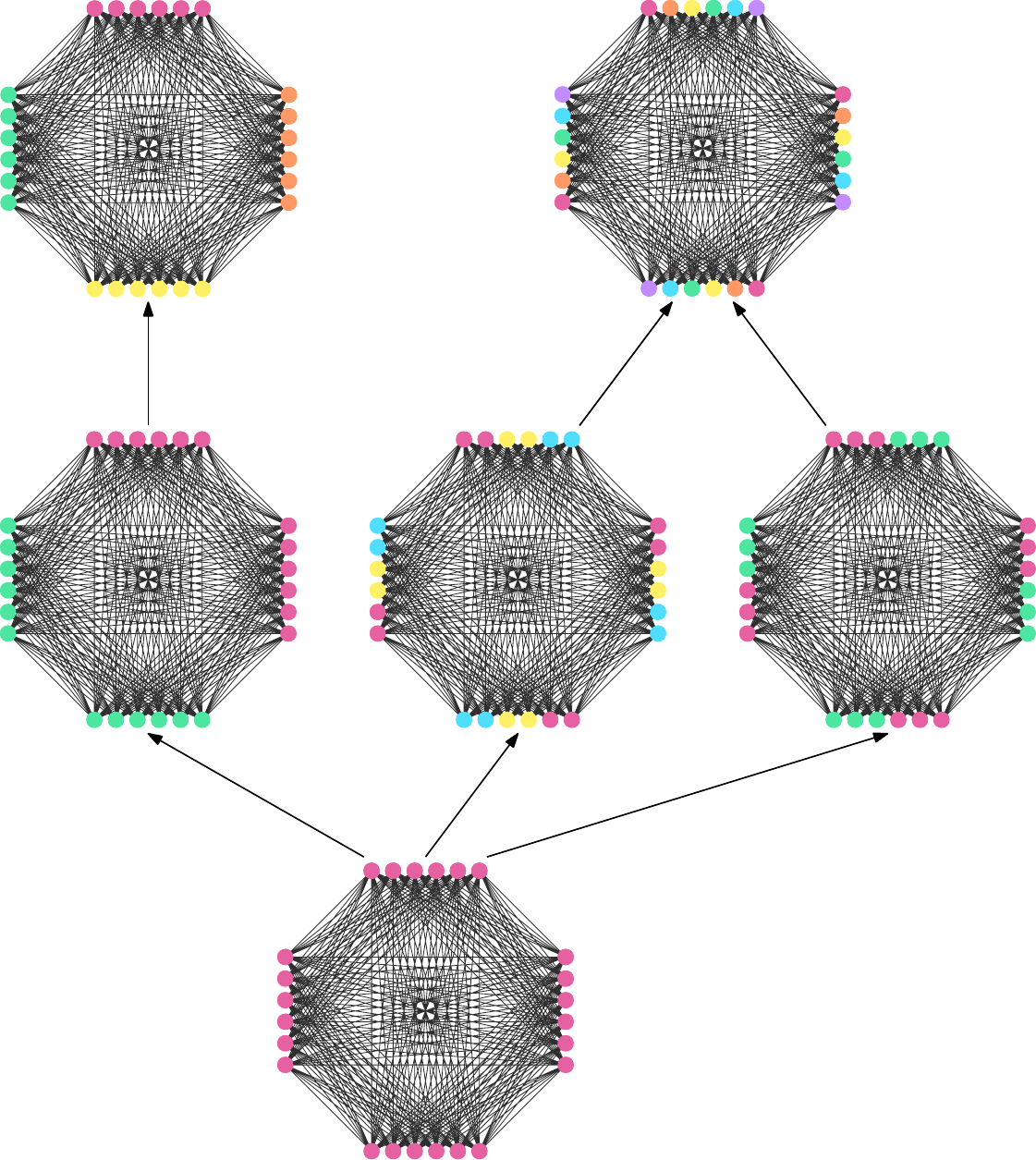}
    \caption{The partial order of all FAT colorings of the regular Turán graph $T(24,4)$.}
    \label{fig:po}
\end{figure}

\begin{example}\label{ex:Turanalpha}
   By  Theorem \ref{thm:merging}, any FAT coloring of $T(N,t)$ is coarser than one of two irreducible FAT colorings, namely, the proper $t$-coloring and the balanced FAT $N/t$-coloring (Figure \ref{fig:po}). The corresponding parameters are $\alpha = 1/(t-1)$ and $\alpha = t/N$, respectively. Hence,  $T(N,t)$ admits a FAT coloring with parameter $\alpha$ for all $\alpha$ in the set
    \[
    \biggl\{ \ell \cdot \frac 1{t-1} \colon \ell\mid t \biggr\} \cup \biggl\{ \ell \cdot \frac tN \colon \ell\mid \frac{N}t\biggr\}.
    \]
\end{example}

Theorem \ref{thm:merging} and \ref{thm:gnigrem} show that, once one knows the irreducible FAT colorings of $G$ and their corresponding parameters $\alpha$, all FAT colorings, and all of their parameters $\alpha$, can be derived from them.
In particular, any FAT coloring using $\chi^{\mathrm{FAT}}$ colors must itself be irreducible.\newline

We conclude this section by listing the graphs whose FAT colorings we have studied so far, together with their irreducible FAT colorings.

\begin{example}[Complete graphs; cf.\ Example \ref{ex:KN}]
    The complete graph $K_N$ on $N$ vertices admits a single irreducible FAT coloring: the proper $N$-coloring. In fact, the FAT colorings of $K_N$ are those that have equal-sized coloring classes, and all of those are coarser than a proper $N$-coloring.
\end{example}
\begin{example}[Trees; cf.\ Example \ref{ex:trees}]
    Every tree, and in particular every star graph (Figure \ref{fig:star}), admits, by Proposition \ref{prop:delta+1}, only one irreducible FAT coloring: the $2$-coloring induced by the bipartition.
\end{example}
\begin{example}[Book graphs; cf.\ Example \ref{ex:bookgraphs}]
    Each $m$-book graph admits an irreducible proper FAT $2$-coloring induced by its bipartition.
    If $m$ is even, it also admits an irreducible FAT $3$-coloring, since in that case $\chi^{\mathrm{FAT}}=3$.
\end{example}
\begin{example}[Cycle graphs; cf.\ Example \ref{ex:cyclegraphs}]
    The cycle graph $C_N$ admits an irreducible proper $2$-coloring when $N$ is even, induced by its bipartition.
    Moreover, if $N\equiv 0\mod 3$, then $C_N$ also admits an irreducible FAT $3$-coloring, since $\chi^{\mathrm{FAT}}=3$.
    Hence, $C_N$ admits two distinct irreducible FAT colorings when $N\equiv0\mod6$, and exactly one otherwise.
\end{example}
\begin{example}[Regular Turán graphs; cf.\ Definition \ref{def:Turangraph}, Theorem \ref{thm:Turan}, and Example \ref{ex:Turanalpha}]
    For the regular Turán graph $T(N,t)$, every irreducible FAT coloring is either the proper $t$-coloring or a balanced $N/t$-coloring, and FAT colorings of both forms are irreducible.
\end{example}

\section{Open questions}\label{section:questions}

Our results suggest several natural directions for further study. We therefore conclude with a list of open problems.

\begin{question}
We have encountered examples where $\chi(G)=\chi^{\mathrm{FAT}}(G)$, as well as cases where one is strictly larger than the other (cf.\ Remark \ref{rmk:chi}). What is the maximum possible gap between the two parameters? In particular, is $\chi^{\mathrm{FAT}}(G)$ always bounded in terms of $\chi(G)$, or can the two diverge arbitrarily?
\end{question}

\begin{question}
What is the computational complexity of determining $\chi^{\mathrm{FAT}}(G)$? Is it NP-complete?
\end{question}

\begin{question}
What is the typical asymptotic behavior of $\chi^{\mathrm{FAT}}(G)$ for random regular graphs?
\end{question}

\begin{question}
Which variants of FAT colorings are worth pursuing? Possible directions include list colorings, fractional versions, or edge colorings.
\end{question}

\begin{question}
By Theorem~\ref{thm:gcd}, FAT chromatic numbers do not satisfy monotonicity: if $H$ is a subgraph of $G$, it does not necessarily follow that $\chi^{\mathrm{FAT}}(H)\leq \chi^{\mathrm{FAT}}(G)$. Under which conditions does monotonicity hold?
\end{question}

\begin{question}
Can Theorem \ref{thm:spectrumAK} be extended to non-regular graphs? In other words, is there an analogue of Theorem \ref{thm:lambda} for the adjacency matrix and for the Kirchhoff Laplacian?
\end{question}

\begin{question}
Can Theorem \ref{thm:Turan}, which characterizes the FAT chromatic number of regular Turán graphs, be extended to all complete multipartite graphs?
\end{question}

\begin{question}
Given a positive integer $b$, does there exist a graph that admits exactly $b$ irreducible FAT colorings?
\end{question}

\begin{question}
    Given any $k\geq1$ and $\alpha\leq1/(k-1)$, can we find a graph admitting a FAT $k$-coloring with parameter $\alpha$?
\end{question}

\begin{question}
    Given any $k\geq1$ and $\alpha\leq1/(k-1)$, can we find a graph with FAT coloring number $k$, admitting a FAT $k$-coloring with parameter $\alpha$?
\end{question}

\subsection*{Acknowledgments}
Some of the ideas presented in this paper were developed during one of our visits to the Alfréd Rényi Institute of Mathematics in Budapest. We are grateful to Ágnes Backhausz for the valuable discussions during that visit.

\section*{Tool and computational resource disclosure}
The authors used a large language model as a tool for language polishing and for identifying possible inconsistencies or errors in preliminary drafts.

\section*{Funding}
Raffaella Mulas is supported by the Dutch Research Council (NWO) through the grant VI.Veni.232.002.

\bibliographystyle{plain} 

\bibliography{Bibliography}

\end{document}